\documentclass{amsart}
\usepackage{amsfonts,amssymb}
\usepackage{enumerate,graphicx}
\usepackage{caption}
\usepackage{subcaption}
\usepackage[square,numbers]{natbib}
\usepackage{mathrsfs,bbm}
\usepackage{tikz,tikzscale,tikz-cd}
\usepackage{multirow,xparse,fp}
\usepackage{environ}
\usepackage{amstext}
\usepackage{hyperref}
\usepackage{diagbox}
 \usepackage[pagewise]{lineno}
\usepackage{comment}
\usepackage{soul} 
\newtheorem{theorem}{Theorem}[]

\newtheorem{lemma}[theorem]{Lemma}
\theoremstyle{definition}

\newtheorem{problem}{Problem}
\newtheorem{remark}[theorem]{Remark}

\allowdisplaybreaks[4]

\begin{document}
\title{The strip entropy approximation of Markov shifts on trees}

\author[Jung-Chao Ban]{Jung-Chao Ban}
\address[Jung-Chao Ban]{Department of Mathematical Sciences, National Chengchi University, Taipei 11605, Taiwan, ROC.}
\address{Math. Division, National Center for Theoretical Science, National Taiwan University, Taipei 10617, Taiwan. ROC.}
\email{jcban@nccu.edu.tw}

\author[Guan-Yu Lai]{Guan-Yu Lai}
\address[Guan-Yu Lai]{Department of Mathematical Sciences, National Chengchi University, Taipei 11605, Taiwan, ROC.}
\email{gylai@nccu.edu.tw}

\author[Cheng-Yu Tsai]{Cheng-Yu Tsai}
\address[Cheng-Yu Tsai]{Department of Mathematical Sciences, National Chengchi University, Taipei 11605, Taiwan, ROC.}
\email{badbearss48@gmail.com}

\thanks{Ban is partially supported by the National Science and Technology Council, ROC (Contract NSTC 111-2115-M-004-005-MY3). Lai is partially supported by the National Science and Technology Council, ROC (Contract NSTC 111-2811-M-004-002-MY2).}
\date{}

\baselineskip=1.2\baselineskip
\maketitle
\begin{abstract}
The strip entropy is studied in this article. We prove that the strip entropy approximation is valid for every ray of a golden-mean tree. This result extends the previous result of [Petersen-Salama, Discrete \& Continuous Dynamical Systems, 2020] on the conventional 2-tree. Lastly, we prove that the strip entropy approximation is valid for eventually periodic rays of a class of Markov-Cayley trees.
\end{abstract}


\section{Introduction}

Let $\mathcal{A}$ be a symbol set with $\left\vert \mathcal{A}\right\vert
=k\in \mathbb{N}$, for any nearest neighbor subshift of finite type (SFT) $X$ on 
$\mathbb{N}^{2}$, i.e., $X\subseteq \mathcal{A}^{\mathbb{N}^{2}}$, one can
define $H_{n}(X)$ to be the set of configurations on $\mathbb{N\times }%
\{1,\ldots ,n\}$ which contain no forbidden patterns of $X$. Such $H_{n}(X)$
can be seen as an $1$-step\footnote{%
That is, the length of the forbidden patterns of $H_{n}(X)$ is $2$.} $%
\mathbb{N}$-SFT with the alphabets in $\mathcal{A}^{1\times n}$, say $\mathcal{A}^{[n]}$. Two letters $a^{[n]}=%
\begin{array}{c}
a_{n} \\ 
\vdots  \\ 
a_{1}%
\end{array}%
$ and $b^{[n]}=%
\begin{array}{c}
b_{n} \\ 
\vdots  \\ 
b_{1}%
\end{array}%
\in \mathcal{A}^{[n]}$ are admissible in $H_{n}(X)$ if and only if $%
\begin{array}{c}
a_{n}b_{n} \\ 
\vdots  \\ 
a_{1}b_{1}%
\end{array}%
\in \mathcal{A}^{2\times n}$ is an admissible pattern of $X$ with size $%
2\times n$ (cf. \cite{pavlov2012approximating}). Suppose $A_{n}(X)$ is an $0$%
-$1$ matrix indexed by $\mathcal{A}^{[n]}$ which is the associated adjacency
matrix of $H_{n}(X)$. The recursive formula from $A_{n}(X)$ to $A_{n+1}(X)$
for $n\in \mathbb{N}$ is found in \cite{ban2005patterns,
pierce2008computing, markley70maximal}. 

In order to illustrate the relation between $X$ and $H_n(X)$, we need more definitions as follows. For $d\geq 1$, let $L\subseteq \mathbb{N}^d$ be a finite subset of $\mathbb{N}^d$ and let $Y$ be an $\mathbb{N}^d$ subshift. The \emph{canonical projection} of $Y\subseteq \mathcal{A}^{\mathbb{N}^d}$ into $\mathcal{A}^L$ is defined by $\mathcal{P}(L,Y)=\{ (y_i)_{i\in L}\in \mathcal{A}^L: y\in Y\}$. The \emph{topological entropy} of $Y$ is defined by
\begin{equation}\label{eq1}
   h_{top}(Y):=\lim_{n\to\infty}\frac{\log \left|\mathcal{P}(L_n,Y)\right|}{\left|L_n\right|}, 
\end{equation}
where $L_n=\{1,...,n\}^d$. The limit (\ref{eq1}) exists and is independent of the choice of F\o{}lner sequence $\{L_n\}_{n=1}^\infty$ (ref. \cite{ceccherini2010cellular}). Suppose $h^{n}(X)=h_{top}(H_n(X))$, it
is known that (cf. \cite{pavlov2012approximating, ban2005patterns,
pierce2008computing}). 
\begin{equation}
\lim_{n\rightarrow \infty }\frac{h^{n}(X)}{n}=h_{top}(X).  \label{2}
\end{equation}

Due to the fact of (\ref{2}), the recursive formula of $\left\{
A_{n}(X)\right\} _{n\geq 2}$ leads us to explore the interesting behavior of
entropy and the finer invariants of X, for example: the dynamic zeta
function \cite{ban2013zeta}. However, for hard square shift\footnote{%
Let $A=\{0,1\}$, the \emph{hard square shift} is the set of configurations in which no two adjacent nodes have identical labels $1$.} $\mathcal{H}$, the value $\frac{h^{n}(\mathcal{H})}{n}$ seems to converge at a linear rate, and Pavlov 
\cite{pavlov2012approximating} proves that $h^{n+1}(\mathcal{H})-h^{n}(%
\mathcal{H})$ converges at least exponentially.

Let $\Sigma =\mathcal{\{}f_{1},\ldots ,f_{d}\mathcal{\}}$, the conventional $%
d$\emph{-tree} is defined as the collection of all possible nodes of length 
$n$ of the symbol set $\Sigma $, i.e., $\Sigma ^{\ast }:=\cup _{n\geq
0}\Sigma ^{n}$, with $\epsilon $ being the only word of zero length and 
\[\Sigma^{n}=\{\omega =(\omega _{1},\ldots ,\omega _{n}):\omega _{i}\in \Sigma~\forall i=1,\ldots ,n\}.\] Denoted by $\left\vert \omega \right\vert $
the \emph{length} of $\omega \in \Sigma ^{\ast }$. A \emph{labeled tree }is
a function $t:\Sigma ^{\ast }\rightarrow \mathcal{A}$, and for each node $%
\omega \in \Sigma ^{\ast }$, $t_{\omega }$ is the label attached to $\omega $%
. Denoted by $\Delta _{n}=\cup _{i=0}^{n}\Sigma ^{i}$, the $n$\emph{%
-block }is a function $u:\Delta _{n}\rightarrow \mathcal{A}$, and a \emph{%
tree-shift} is a set $\mathcal{T}\subseteq \mathcal{A}^{\Sigma ^{\ast }}$ of
labeled trees which avoid all of a certain set of forbidden blocks. The tree-shifts concept, introduced by Aubrun and Beal \cite{aubrun2012tree}, is drawing considerable attention recently owing to the abundance of interesting phenomena (cf. \cite{petersen2021asymptotic, ban2017mixing, ban2017tree,
ban2021structure, petersen2020entropy}) related to shifts defined on the amenable group (cf. \cite{lemp2017shift, ceccherini2010cellular}).

For $\mathbf{A}=\left( A_{1},\ldots ,A_{d}\right) $, where $A_{i}\in
\{0,1\}^{k\times k}$ $\forall i=1,\ldots ,d$, the associated \emph{Markov
tree-shift}, say $\mathcal{T}_{\mathbf{A}}$, is defined as 
\[\mathcal{T}_{%
\mathbf{A}}=\{t\in \mathcal{A}^{\Sigma ^{\ast }}:A_{i}(t_{\omega },t_{\omega
i})=1~\forall \omega \in \Sigma ^{\ast },i\in \Sigma \}.\] A tree-shift $%
\mathcal{T}_{\mathbf{A}}$ is called a \emph{Markov hom tree-shift} if $%
A_{1}=\cdots =A_{d}=:A$, and we write $\mathcal{T}_{\mathbf{A}}=\mathcal{T}%
_{A}$ on that occasion. We call $\mathcal{T}_{\mathbf{A}}$ a \emph{Markov
anisotropic tree-shift} if $\mathcal{T}_{\mathbf{A}}$ is not a Markov hom tree-shift(cf. \cite{chandgotia2016mixing,petersen2020entropy}). For a $d$-tree $\Sigma ^{\ast }$, the associated \emph{boundary} (cf. \cite{benjamini1994markov})%
, say $\partial \Sigma ^{\ast }$ is defined as $\partial \Sigma ^{\ast
}=\{(\epsilon ,f_{i_{1}},f_{i_{1}}f_{i_{2}},\ldots ):f_{i_{j}}\in \Sigma $ $%
\forall j\geq 1\}$, i.e., the collection of \emph{rays}\footnote{%
This is, the set of infinite non-self-intersection paths emanating from $%
\epsilon $.} in $\Sigma ^{\ast }$. 

Suppose $\mathcal{T}\subseteq \mathcal{A}%
^{\Sigma ^{\ast }}$ is a tree-shift and $F\subseteq \Sigma ^{\ast }$, we
denote by $\mathcal{P}(F,\mathcal{T}):\mathcal{A}^{\Sigma ^{\ast
}}\rightarrow \mathcal{A}^{F}$ the \emph{canonical projection }of $\mathcal{%
T\subseteq A}^{\Sigma ^{\ast }}$ into $\mathcal{A}^{F}$, i.e., $\mathcal{P}%
(F,\mathcal{T})=\{(t_{\omega})_{\omega\in F}\in \mathcal{A}^{F}: t\in 
\mathcal{T}\}$. The \emph{entropy }of $\mathcal{T}$ is defined as the growth
rate of the patterns along the blocks $\Delta _{n}$ as $n\rightarrow \infty $%
. Precisely, 
$$h(\mathcal{T}):=\lim_{n\rightarrow \infty }\frac{\log \lvert\mathcal{%
P}(\Delta_n,\mathcal{T})\rvert}{\left\vert \Delta _{n}\right\vert },$$ the existence of
the limit is presented in \cite{PS-2017complexity, petersen2020entropy}.

\begin{figure} 
	\centering 
	\includegraphics[width=0.8\textwidth]{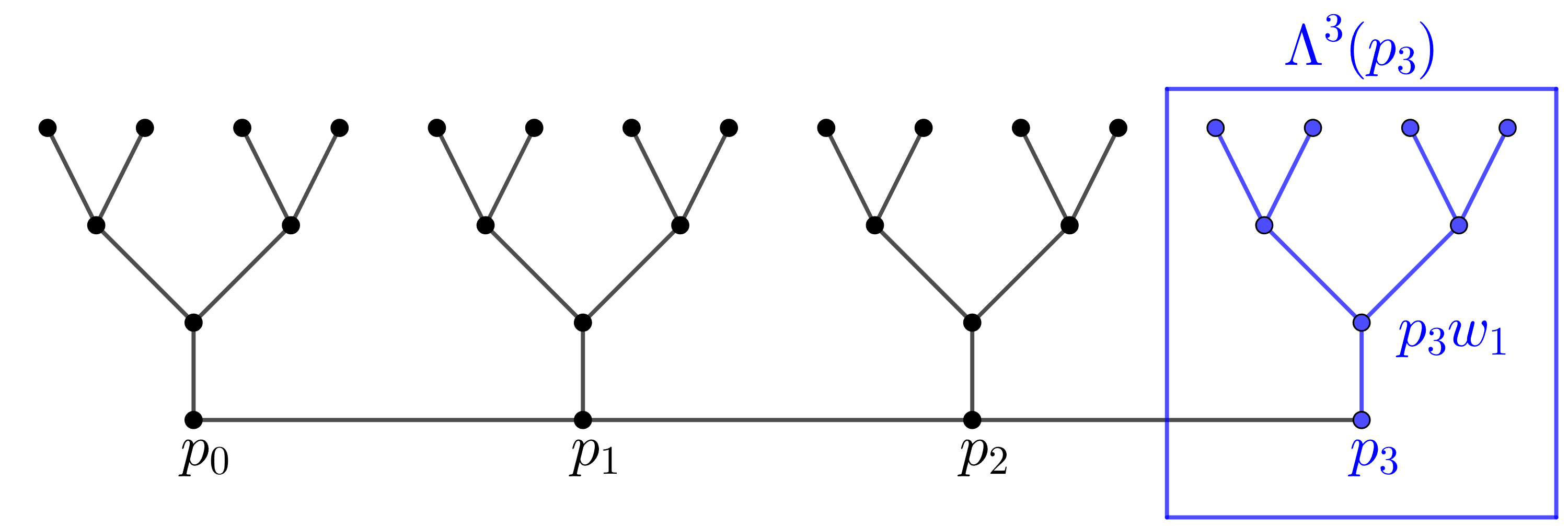} 
	\caption{The graph of ${\color{blue}\Lambda^3(p_3)}$ and $\Lambda_4^3(p)$ in the 2-tree along the path $p=(p_0,p_1,p_2,...).$ } 
	\label{2-strip tree}
\end{figure}

Let $p\in \partial \Sigma^{\ast }$, say $p=(\epsilon
,f_{i_{1}},f_{i_{1}}f_{i_{2}},\ldots)=(p_{0},p_{1},p_{2},\ldots )$, where $p_{0}=\epsilon$, and $p_{n}=f_{i_{1}}f_{i_{2}}\cdots f_{i_{n}}$. Given $m\in \mathbb{N}$ and for each $0\leq j\leq m-1$, define \[\Lambda^{n}(p_{j})=\{p_{j}\omega\in\Sigma^{\ast}:\lvert p_{j}\omega\rvert\leq j+n-1~\text{and}~\omega_{1}\neq p_{j+1}\}\] and $\Lambda_m^n(p)=\cup_{j=0}^{m-1}\Lambda^{n}(p_{j})$ (cf. Figure \ref{2-strip tree}). Finally, we define
\begin{equation}
h^{n}(\mathcal{T},p)=\limsup_{m\rightarrow
\infty }\frac{\log \lvert\mathcal{P}(\Lambda _{m}^{n}(p),\mathcal{T})\rvert}{\left\vert
\Lambda _{m}^{n}(p)\right\vert }
\end{equation}
and the authors \cite{petersen2020entropy} call $h^{n}(\mathcal{T},p)$ \emph{site specific strip entropies} \emph{along }$p$ ($n$-\emph{strip entropy} \emph{along }$p$, in brief) and the path $p$ is called the \emph{base}. Let $G=\left[ 
\begin{array}{cc}
1 & 1 \\ 
1 & 0%
\end{array}\right] $, Petersen and Salama \cite{petersen2020entropy} extend the aforementioned result to the hard square shift on a $2$-tree, i.e., a free
semigroup with two generators ($f_1$ and $f_2$). They obtain that 
\begin{equation}
\lim_{n\rightarrow \infty }h^{n}(\mathcal{T}_{G},p)=h(\mathcal{T}_{G})\text{,%
}  \label{3}
\end{equation}%
where $p=(\epsilon ,f_{1},f_{1}^{2},f_{1}^{3},\ldots )$, where $f_{1}^{r}:=%
\overbrace{f_{1}f_{1}\cdots f_{1}}^{r\text{-times}}$, that is, the leftmost
ray in a $2$-tree. The result can be seen as an analogous result to (\ref{2}) since if we take $\mathbb{N}$ as a horizontal ray in $\mathbb{N}^{2}$,
then 
\[
h^{n}(X,\mathbb{N})=\lim_{n\rightarrow \infty }\frac{\log \mathcal{P}%
(\Lambda _{m}^{n}(\mathbb{N}),X)}{\left\vert \Lambda _{m}^{n}(\mathbb{N}%
)\right\vert }=\lim_{m\rightarrow \infty }\frac{\log \Sigma _{A_{n}}^{m}}{%
n\times m}=\frac{1}{n}h^{n}(X)\rightarrow h(X)\text{.}
\]

The purpose of this article is to extend the aforementioned result to a
broad class of lattices, namely, Markov-Cayley trees. Let $M\in
\{0,1\}^{d\times d}$ and 
\[\Sigma _{M}^{n}=\{\omega =(\omega _{1},\omega
_{2},\ldots ,\omega _{n}):M(\omega _{i},\omega _{i+1})=1~\forall
i=1,\ldots ,n-1\},\] then $\Sigma _{M}^{\ast }=\cup _{n\geq 0}\Sigma _{M}^{n}$
is called a \emph{Markov-Cayley tree}. Clearly, the conventional $d$-tree is a kind of Markov-Cayley tree with $M$ being a $0$-$1$ full matrix $E$\footnote{%
That is, $E$ is a matrix with all entries being $1$'s.}, and it is
not difficult to show that (\ref{3}) holds true for any $p\in \partial
\Sigma ^{\ast }$. Since for any ray $q\in \partial \Sigma ^{\ast }$ other
than $p=(\epsilon ,f_{1},f_{1}^{2},f_{1}^{3},\ldots )$, due to the fact that if 
$\Sigma ^{\ast }=\Sigma _{E}^{\ast }$, it can be easily checked that $%
\Lambda _{m}^{n}(p)=\Lambda _{m}^{n}(q)$ for $m,n\in \mathbb{N}$. That is,
the limit of $h^{n}(\mathcal{T},p)$ is independent
of the choice of $p\in \Sigma ^{\ast }$. However, if $M$ is not a full
matrix, the value $\lim_{n\rightarrow \infty }h^{n}(\mathcal{T},p)$ (or $%
\Lambda _{m}^{n}(p)$) strongly depends on the choice of the ray $p\in
\partial \Sigma _{M}^{\ast }$. A natural question arises: Given a
Markov-Cayley tree $\Sigma _{M}^{\ast }$, which $p\in \partial \Sigma
_{M}^{\ast }$ allows the admission strip entropy approximation (\ref{3})? The problem
is extremely hard since the irregular behavior of the $\Lambda _{m}^{n}(p)$
makes it more difficult to count the number of the `$2$-dimensional'
patterns of a tree-shift along the base $p$. Nevertheless, Theorem \ref{Thm: 1} below
unveils that the above approximation (\ref{3}) is still valid for a broad
class of $\mathcal{T}_{A}$ on a golden-mean tree\footnote{%
That is, a Markov-Cayley tree with $M=G$.} with $2$ generators. Theorem \ref{Thm: 1-2} reveals that we may still use the eventually periodic ray as a base to make a strip entropy approximation. Finally, we remark that if $A=G$, then (\ref{1}) extends the result of (\ref%
{3}) to a golden-mean tree (see also Remark \ref{rmk4}).

\begin{theorem}\label{Thm: 1}
For any primitive $A\in \{0,1\}^{k\times k}$ and $p\in \partial
\Sigma _{G}^{\ast }$, we have 
\begin{equation}
\lim_{n\rightarrow \infty }h^{n}(\mathcal{T}_{A},p)=h(\mathcal{T}_{A}). \label{1}
\end{equation}
\end{theorem}
Let $M\in \{0,1\}^{d\times d}$ and $\Sigma _{M}^{\ast }$ be the associated
Markov-Cayley tree. For $u=(u_{1},u_{2},\ldots )\in \Sigma _{M}^{\ast }$,
and $n\in \mathbb{N}$, define $u[1,n]:=(u_{1},\ldots ,u_{n})\in \Sigma
_{M}^{n}$. For $u,v\in \Sigma _{M}^{\ast }$ we denote by $v\leq u$ if $%
\left\vert v\right\vert \leq \left\vert u\right\vert $ and $u[1,\left\vert
v\right\vert ]=v$. A finite collection $S$ of some vertices of $\Sigma
_{M}^{\ast }$ is called a \emph{complete prefix set }(CPS)\emph{\ }if for
every $u\in \Sigma _{M}^{\ast }$ with $\left\vert u\right\vert \geq \max
\{\left\vert v\right\vert :v\in S\}$, there exists $v\in S$ such that $v\leq
u$ and $v\nleq v^{\prime}$ for all $v$ and $v^{\prime }\in S$. Suppose $%
u\in \Sigma _{M}^{\ast }$, the \emph{follower set }of $u$ is defined as  
\[
F_{M}(u)=\{v\in \Sigma _{M}^{\ast }:uv\in \Sigma _{M}^{\ast }\}\text{.}
\]%
Finally, a Markov-Cayley tree $\Sigma _{M}^{\ast }$ is called \emph{complete
recursive tree }if there is a CPS $S$ of $\Sigma _{M}^{\ast }$ such that $%
F_{M}(u)=\Sigma _{M}^{\ast }$ for all $u\in S$. The concept of a complete recursive tree is introduced in \cite{ban2022Commutativity} and the authors prove that the commutativity of the entropy holds true on such a tree. Meanwhile, the characterization of a complete recursive tree is also presented therein (Theorem \ref{Thm: complete}). Finally, we extend Theorem \ref{Thm: 1} to a class of complete recursive trees on each eventually periodic path $p\in \partial \Sigma _{M}^{\ast }$, as shown below.   

\begin{theorem}\label{Thm: 1-2}
Let $A\in \{0,1\}^{k\times k}$ be a primitive and $\Sigma _{M}^{\ast }$ be a
complete recursive tree with $M$ has exactly one full row. Then (\ref{1}) holds true for every eventually
periodic ray $p\in \partial \Sigma _{M}^{\ast }$.
\end{theorem}

In closing, we raise two problems related to our findings.

\begin{problem}
\label{Prob: 1}

\begin{itemize}
 \item[1.] What type of $p\in \partial \Sigma _{M}^{\ast }$ yields a valid
approximation (\ref{1}) for general $M\in \{0,1\}^{d\times d}$ and $A\in
\{0,1\}^{k\times k}$?

\item[2.] What is the rate of the convergence (\ref{1})?   
\end{itemize}
\end{problem}

For the Problem \ref{Prob: 1} (2), recall that in $\mathbb{N}^{2}$ SFT, the convergence rate (\ref{1}) is linear. However, based on some numerical data, we speculate that the rate of the convergence (\ref{1}) for a Markov tree-shift is at least exponential. In the remainder of this article, we provide the complete proof for Theorem \ref{Thm: 1} (resp. Theorem \ref{Thm: 1-2}) in Section \ref{sub 2} (resp. Section \ref{sub 3}).

\section{Proof of Theorem \ref{Thm: 1}}\label{sub 2}
In this section, we give the proof of Theorem \ref{Thm: 1}. Before we prove the Theorem \ref{Thm: 1}, the following notations are needed.

Let $\mathcal{A}=\{1,...,k\}$ be the set of symbols and let $A\in \{0,1\}^{k\times k}$ be a primitive 0-1 $k\times k$ matrix. Let $\Sigma_M^*$ be the Markov tree and let $\mathcal{T}_A$ be the associated tree-shift on $\Sigma_M^*$. For $p=(p_0,p_1,p_2,...)\in\partial\Sigma^*_M$ with $p_0=\epsilon$, let $\lambda^n(p_m)=|\Lambda^n(p_m)|$ be the cardinality of $\Lambda^n(p_m)$. Let $\beta_n=|\mathcal{P}(\triangle_n,\mathcal{T}_A)|$ be the cardinality of the set that contains all labelings of $\mathcal{T}_A$ on $\triangle_n$ and let 
\[\beta_n(i)=\lvert\mathcal{P}(\triangle_n,\{t=(t_w)\in\mathcal{T}_A:t_\epsilon=i\})\rvert\] 
be the cardinality of the set that contains all labelings of $\mathcal{T}_A$ on $\triangle_n$ with the root $\epsilon$ being labeled with $i\in \mathcal{A}$. Let $\alpha_{n,m}=|\mathcal{P}(\Lambda_m^n(p),\mathcal{T}_A)|$ be the cadinality of the set that contains all labelings of $\mathcal{T}_A$ on $\Lambda_m^n(p)$ and let 
\[\alpha_{n,m}(i)=|\mathcal{P}(\Lambda_m^n(p),\{t=(t_w)\in\mathcal{T}_A:t_{p_m}=i\})|\] be the cadinality of the set that contains all labelings of $\mathcal{T}_A$ on $\Lambda_m^n(p)$ with the node $p_m$ being labeled $i\in\mathcal{A}$.

We now consider $\Sigma_M^*=\Sigma_G^*,$ i.e., a golden-mean tree, and the adjacency matrix  $A=(a_{ij})_{1\leq i,j\leq k}$ is a primitive $k\times k~0-1$ matrix. Before we prove the Theorem \ref{Thm: 1}, we classify the possible types of $\Lambda^n(p_j)$ and $\Lambda^n(p_j)\cup\Lambda^n(p_{j+1})$ as follows. For each path $p=(p_0,p_1,p_2,...)\in\partial\Sigma^*_G$, the type of $\Lambda^n(p_j)$ is a function of $(f_s,f_{s'})$ that satisfies $(p_j,p_{j+1})=(p_{j-1}f_s,p_{j-1}f_sf_{s'})$ for all $f_sf_{s'}\in \Sigma_G^*.$ Due to the constraint of the golden-mean tree, we say that $\Lambda^n(p_j)$ is of type $T_1$ if $(p_j,p_{j+1})=(p_{j-1}f_1,p_{j-1}f_1^2)$; it is of type $T_2$ if $(p_j,p_{j+1})=(p_{j-1}f_1,p_{j-1}f_1f_2)$; and it is of type $T_3$ if $(p_{j},p_{j+1})=(p_{j-1}f_2,p_{j-1}f_2f_1)$. This implies that there are exactly three types $T_1, T_2$ and $T_3$ of $\Lambda^n(p_j)$ (cf. Figure \ref{3types}). Similarly, the type of $\Lambda^n(p_{j})\cup\Lambda^n(p_{j+1})$ is a function of $(p_j,p_{j+1},p_{j+2})$. Similarly, we say that $\Lambda^n(p_{j})\cup\Lambda^n(p_{j+1})$ is of type $S_1$ if $(p_j,p_{j+1},p_{j+2})=(p_{j-1}f_1,p_{j-1}f_1^2,p_{j-1}f_1^3);$ it is of type $S_2$ if $(p_j,p_{j+1},p_{j+2})=(p_{j-1}f_1,p_{j-1}f_1^2,p_{j-1}f_1^2f_2);$ it is of type $S_3$ if $(p_j,p_{j+1},p_{j+2})=(p_{j-1}f_2,p_{j-1}f_2f_1,p_{j-1}f_2f_1^2);$ it is of type $S_4$ if $(p_j,p_{j+1},p_{j+2})=(p_{j-1}f_1,p_{j-1}f_1f_2,p_{j-1}f_1f_2f_1);$ it is of type $S_5$ if $(p_j,p_{j+1},p_{j+2})=(p_{j-1}f_2,p_{j-1}f_2f_1,p_{j-1}f_2f_1f_2).$ This implies that there are exactly five types $S_1,S_2,\cdots, S_5$ of $\Lambda^n(p_j)\cup\Lambda^n(p_{j+1})$ (cf. Figure \ref{3types}).

\begin{figure} 
	\centering 
	\includegraphics[width=1\textwidth]{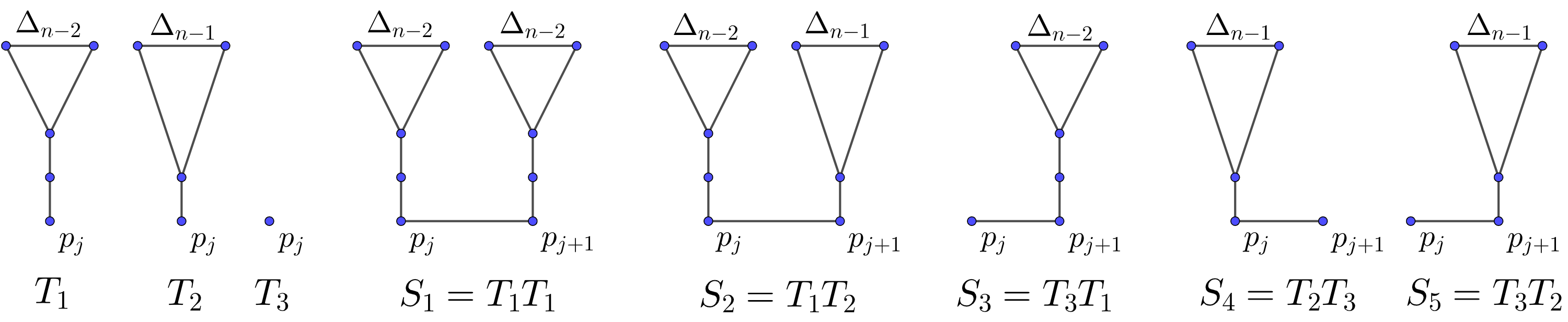} 
	\caption{The graph of three types $(T_1,T_2,T_3)$ of $\Lambda^n(p_j)$ and five types $(S_1,...,S_5)$ of $\Lambda^n(p_j)\cup\Lambda^n(p_{j+1})$ in $\mathcal{T}_A$ along $p.$} 
	\label{3types} 
\end{figure}

The following lemma establishes the recursive relation between $\alpha_{n,m+1}(1),...,\alpha_{n,m+1}(k)$ and $\alpha_{n,m}(1),...,\alpha_{n,m}(k)$ for $m\geq 0$. To shorten statements and notations, we define some properties and operations of matrices. For matrices $H_1=(h_{st}^1)_{1\leq s,t\leq k}$ and $H_2=(h_{st}^2)_{1\leq s,t\leq k}$ , we define a \emph{Hadarmard product} $H_1\circ H_2=H_3,$ where $H_3=(h_{st}^3)_{1\leq s,t\leq k}$ is a $k\times k$ matrix with $h_{st}^3=h_{st}^1h_{st}^2.$ Let $B_1=(b_{st}^1)_{1\leq s,t\leq k}$ and $B_2=(b_{st}^2)_{1\leq s,t\leq k}$ with $b_{st}^1=\sum_{j=1}^k(\sum_{\ell=1}^k a_{s\ell}a_{\ell j})\beta_{n-2}(j)$ and $b_{st}^2=\sum_{\ell=1}^ka_{s\ell}\beta_{n-1}(\ell)$. Let $A_1=A^T\circ B_1,A_2=A^T\circ B_2$ and $A_3=A^T,$ where $A^T$ is denoted by the transpose of $A.$ We remark that $A_1,A_2$ and $A_3$ are primitive, since $A$ is primitive and $B_1,B_2>0.$

\begin{lemma}\label{lemma 1}
For any $m\geq 0,$ there exists a $k\times k$ matrix $R\in\{A_1,A_2,A_3\}$ such that 
\[(\alpha_{n,m+1}(1),...,\alpha_{n,m+1}(k))^T=R(\alpha_{n,m}(1),...,\alpha_{n,m}(k))^T.\]
\end{lemma}

\begin{proof}
For $m\geq 0$, let $s\in\{1,...,k\}$ be the labelling of the node $p_{m+1}.$ From the definition of $\alpha_{n,m+1}(s),$ we get $\alpha_{n,m+1}(s)=P_1P_2,$ where
\[P_1=\left|\mathcal{P}\left(\Lambda_m^n(p)\cup \{p_{m+1}\},\{t=(t_w)\in\mathcal{T}:t_{p_{m+1}}=s\}\right)\right|\]
is the cardinality of the set that contains all labellings of $\mathcal{T}$ on $\Lambda_m^n(p)\cup \{p_{m+1}\}$ with $t_{p_{m+1}}=s,$ and
\[P_2=\left|\mathcal{P}\left(\Lambda^n(p_{m+1}),\{t=(t_w)\in\mathcal{T}:t_{p_{m+1}}=s\}\right)\right|\]
is the cardinality of the set that contains all labellings of $\mathcal{T}$ on $\Lambda^n(p_{m+1})$ with $t_{p_{m+1}}=s.$ Now, we calculate the values of $P_1$ and $P_2$ in each type of $S_1,...,S_5.$  First, we claim that
\begin{equation}
    P_1=a_{1s}\alpha_{n,m}(1)+a_{2s}\alpha_{n,m}(2)+\cdots +a_{ks}\alpha_{n,m}(k)
\label{P1}
\end{equation}
holds in each type of $S_1,...,S_5.$ Indeed, in each type of $S_1,..., S_5,$ the cardinality of the set that contains all labelings of $\mathcal{T}$ on $\Lambda_m^n(p)\cup \{p_{m+1}\}$ is determined by the value of $\alpha_{n, m}(j)$ for which symbol $j\in\{1,...,k\}$ can attach to the symbol $s.$  Since $A$ is the adjacency matrix, if the symbol $j$ can attach to the symbol $s,$ then $a_{j s}=1$; otherwise $a_{j s}=0$ for all $j=1,...,k.$ This implies the equation (\ref{P1}) holds. Second, we claim that in the types $S_1$ and $S_3,$
\begin{equation}
    P_2=\sum_{\ell=1}^k a_{s\ell}a_{\ell 1}\beta_{n-2}(1)+\sum_{\ell=1}^k a_{s\ell}a_{\ell 2}\beta_{n-2}(2)+\cdots +\sum_{\ell=1}^k a_{s\ell}a_{\ell k}\beta_{n-2}(k)
\label{P2-1}
\end{equation}
in the types $S_2$ and $S_5,$
\begin{equation}
    P_2=a_{s1}\beta_{n-1}(1)+a_{s2}\beta_{n-1}(2)+\cdots +a_{sk}\beta_{n-1}(k),
\label{P2-2}
\end{equation}
and in the type $S_4$, we have $P_2=1$. Indeed, in the types $S_1$ and $S_3,$ the cardinality of the set that contains all labelings of $\mathcal{T}$ on $\Lambda^n(p_{i+1})$ is determined by the value of $\beta_{n-2}(j)$ for which symbol $j\in\{1,...,k\}$ can be attached by $s$ with length $3.$ Consider $A^2,$ where the cardinality of the set that contains all labelings such that $s$ attaches to $j$ with length $3$ is $\sum_{\ell=1}^k a_{s\ell}a_{\ell j}.$ This implies the equation (\ref{P2-1}) holds. In the types $S_2$ and $S_5,$ the cardinality of the set that contains all labelings of $\mathcal{T}$ on $\Lambda^n(p_{i+1})$ is determined by the value of $\beta_{n-1}(j)$ for which symbol $j\in\{1,...,k\}$ can be attached by $s,$ i.e., $a_{s j}=1.$ This implies (\ref{P2-2}) holds. In the type $S_4,$ the value of $P_2=1$ is clear.

Finally, combining the equality (\ref{P1}) with (\ref{P2-1}), we obtain $R=A^T\circ B_1,$ where $B_1=(b_{st}^1)_{1\leq s,t\leq k}$ with $b_{st}^1=\sum_{j=1}^k(\sum_{\ell=1}^k a_{s\ell}a_{\ell j})\beta_{n-2}(j)$ in the types $S_1$ and $S_3.$ Combining the equality (\ref{P1}) with (\ref{P2-2}), we obtain $R=A^T\circ B_2,$ where $B_2=(b_{st}^2)_{1\leq s,t\leq k}$ with $b_{st}^2=\sum_{\ell=1}^k a_{s\ell}\beta_{n-1}(\ell)$ in the types $S_2$ and $S_5.$ Combining the equality (\ref{P1}) with the fact that $P_2=1$, we obtain $R=A^T$ in the type $S_4.$ The proof is complete. 
\end{proof}

\begin{proof}[Proof of Theorem \ref{Thm: 1}] Due to the constraint of the golden-mean tree $\Sigma_G^*$, without loss of generality, we may assume the path $p\in\partial \Sigma_G^*$ is of the form as follows.  
$$p=\{\epsilon , f_2,f_2f_1,... ,f_2f_1^{n_1}, f_2f_1^{n_1}f_2, f_2f_1^{n_1}f_2f_1,...,f_2f_1^{n_1}f_2f_1^{n_2},...\}$$
where $n_i\in\mathbb{N}$ for all $i\in\mathbb{N}.$ Applying Lemma \ref{lemma 1}, there is a sequence of matrices $\{R_n\}_{n=1}^\infty$ with
\begin{align*}
    R_n=\left\{\begin{array}{ll}
        A_3& \mbox{, if }n=\sum_{i=1}^m (n_i+1)+1 \mbox{ for some }m\geq 0,     \\
        A_2&  \mbox{, if }n=\sum_{i=1}^m (n_i+1) \mbox{ for some }m\geq 1,   \\
        A_1&   \mbox{, otherwise.}  \\
    \end{array}\right.
\end{align*}
That is, $\{R_1,R_2,...\}=\{A_3,\overbrace{A_1,..., A_1}^{n_1-1},A_2,A_3,\overbrace{A_1,..., A_1}^{n_2-1},A_2,...\}$.

For $m_t=\sum_{i=1}^t (n_i-1)+2t+1,$ we have
$$\alpha_{n,m_t}=\sum_{i,j=1}^k(A_3A_1^{n_1-1}A_2A_3A_1^{n_2-1}A_2\cdots A_3A_1^{n_t-1}A_2)_{ij}.$$
Since $A_3$ and $A_2$ are primitive, we have $A_3A_2$ is primitive. Hence, there exist constants $c_0,d_0>0$ such that 
\[c_0\rho_{A_3A_2}^n\leq \sum_{i,j=1}^k(A_3A_2)^n_{ij}\leq d_0\rho_{A_3A_2}^n,\]
for all $n\in\mathbb{N}$. 

Let $\xi_n^{T_i}$ be the number of nodes of the type $T_i$ for $i=1,2,3$ (cf. Figure 1). First, we claim that
\begin{equation}\label{ev}
    \lim\limits_{n\rightarrow \infty}\frac{\log\rho_{A_1}}{\xi_n^{T_1}}=\lim\limits_{n\rightarrow \infty}\frac{\log\beta_{n-2}}{\xi_n^{T_1}}.
\end{equation}
Indeed, since $A_1=A^T\circ B_1,$ where $B_1=(b_{st}^1)_{1\leq s,t\leq k}$ with $b_{st}^1=\sum_{j=1}^k(\sum_{\ell=1}^k a_{s\ell}a_{\ell j})\beta_{n-2}(j),$ we have that 
\begin{equation}\label{CD}
    \beta_{n-2}L\leq A_1\leq \beta_{n-2}U,
\end{equation}
where $L=A^T\circ L'$ and $U=A^T\circ U',$ where $L'=(r)_{1\leq s,t\leq k}$ with $0<r<\min\left\{\frac{\beta_{n-2}(1)}{\beta_{n-2}},...,\frac{\beta_{n-2}(k)}{\beta_{n-2}}\right\}$ and $U'=(k)_{1\leq s,t\leq k}.$ Note that the constant $r$ exists since the adjacency matrix $A$ is primitive. Hence, $\beta_{n-2}\rho_{L}\leq \rho_{A_1}\leq \beta_{n-2}\rho_U.$ This implies the equality (\ref{ev}) holds.

Second, we claim that the upper bound of $\frac{\log \alpha_{n,m_t}}{\lambda^n(p_0)+\cdots +\lambda^n(p_{m_t})}$ is of the form
\begin{equation}
    \frac{\log \alpha_{n,m_t}}{\lambda^n(p_0)+\cdots +\lambda^n(p_{m_t})}\leq \frac{\left(\sum_{j=1}^tn_j-t\right)\log(k\rho_F\beta_{n-2})+t\log \rho_{A_3A_2}+\log d_0}{\left(\sum_{j=1}^tn_j-t\right) \xi_n^{T_1}+t(\xi_n^{T_2}+\xi_n^{T_3})+\xi_n^{T_2}},
\label{upper}
\end{equation}
Indeed, applying the inequality (\ref{CD}), we obtain $A_1\leq k\beta_{n-2}F,$ where $F$ is the full matrix. Then $F$ is a diagonal matrix. Applying the Jordan decomposition of $F,$ we obtain that for each $n\in\mathbb{N}$, $\frac{F^n}{\rho_F^n}\leq I$, where $I$ is the identity $k\times k$ matrix. Thus,
\begin{align*}
\frac{\log \alpha_{n,m_t}}{\lambda^n(p_0)+\cdots +\lambda^n(p_{m_t})}&\leq \frac{\log(k\beta_{n-2})^{n_1+n_2+\cdots +n_t-t}\sum_{i,j=1}^k(A_3F^{n_1-1}A_2\cdots A_3F^{n_t-1}A_2)_{ij}}{\lambda^n(p_0)+\cdots +\lambda^n(p_{m_t})}\\
&\leq \frac{\log(k\rho_F\beta_{n-2})^{n_1+n_2+\cdots +n_t-t}\sum_{i,j=1}^k(A_3A_2)^t_{ij}}{\lambda^n(p_0)+\cdots +\lambda^n(p_{m_t})}\\
&\leq \frac{\log(k\rho_F\beta_{n-2})^{n_1+n_2+\cdots +n_t-t} d_0\rho_{A_3A_2}^t}{\lambda^n(p_0)+\cdots +\lambda^n(p_{m_t})}\\
&=\frac{[n_1+n_2+\cdots +n_t-t]\log (k\rho_F\beta_{n-2})+t\log \rho_{A_3A_2}+\log d_0}{(n_1+n_2+\cdots +n_t-t) \xi_n^{T_1}+t(\xi_n^{T_2}+\xi_n^{T_3})+\xi_n^{T_2}}.
\end{align*}
 
 Third, we claim that the lower bound of $\frac{\log \alpha_{n,m_t}}{\lambda^n(p_0)+\cdots +\lambda^n(p_{m_t})}$ is of the form
\begin{equation}
    \frac{\log \alpha_{n,m_t}}{\lambda^n(p_0)+\cdots +\lambda^n(p_{m_t})}\geq \frac{\left(\sum_{j=1}^tn_j-t\right)\log \rho_{A_1}+t(\log c_n\rho_{A_3A_2})+\log c_0}{\left(\sum_{j=1}^tn_j-t\right)\xi_n^{T_1}+t(\xi_n^{T_2}+\xi_n^{T_3})+\xi_n^{T_2}},
\label{lower}
\end{equation}
where $v_{A_1}~(w_{A_1})$ is the right (left) eigenvector corresponding to $\rho_{A_1}$ and $c_n=\min\{(v_{A_1})_i(w_{A_1}^T)_j\}.$ Indeed, by Theorem 4.5.12 \cite{LM-1995}, we have that for each $n\in\mathbb{N},$
\begin{equation}\label{vw}
v_{A_1}w_{A_1}^T\leq \frac{(A_1)^n}{\rho_{A_1}^n}.
\end{equation}
Applying inequality (\ref{vw}), we obtain
\begin{align*}
\frac{\log \alpha_{n,m_t}}{\lambda^n(p_0)+\cdots +\lambda^n(p_{m_t})}&\geq \frac{\log \sum_{i,j=1}^k(A_3(v_{A_1}w_{A_1}^T)^{n_1-1}A_2\cdots A_3(v_{A_1}w_{A_1}^T)^{n_t-1}A_2)_{ij}}{\lambda^n(p_0)+\cdots +\lambda^n(p_{m_t})}\\
&\geq \frac{\log c_n^k\rho_{A_1}^{n_1+\cdots +n_k-k}\sum_{i,j=1}^k(A_3F^{n_1-1}A_2\cdots A_3F^{n_k-1}A_2)_{ij}}{\lambda^n(p_0)+\cdots +\lambda^n(p_{m_t})}\\
&\geq  \frac{\log c_n^t\rho_{A_1}^{n_1+\cdots +n_t-t}\sum_{i,j=1}^{k}(A_3A_2)^t_{ij}}{\lambda^n(p_0)+\cdots +\lambda^n(p_{m_t})}\\
&\geq  \frac{\log c_n^t\rho_{A_1}^{n_1+\cdots +n_t-t}c_0\rho_{A_3A_2}^t}{\lambda^n(p_0)+\cdots +\lambda^n(p_{m_t})}\\
&=  \frac{(n_1+\cdots +n_t-t)\log \rho_{A_1}+t\log c_n\rho_{A_3A_2}+\log c_0}{(n_1+n_2+\cdots +n_t-t)\xi_n^{T_1}+t(\xi_n^{T_2}+\xi_n^{T_3})+\xi_n^{T_2}}.
\end{align*}

 Finally, combining the inequality (\ref{upper}) with (\ref{lower}), we obtain 
\begin{equation}\label{UL}
  \frac{\log\rho_{A_1}+\log c_n\rho_{A_3A_2}}{\xi_n^{T_1}+\xi_n^{T_2}+\xi_n^{T_3}}\leq \limsup_{m\rightarrow\infty} \frac{\log \alpha_{n,m_t}}{\lambda^n(p_0)+\cdots +\lambda^n(p_{m_t})}\leq \frac{\log(k\rho_F\beta_{n-2})+\rho_{A_3A_2}}{\xi_n^{T_1}+\xi_n^{T_2}+\xi_n^{T_3}}
\end{equation}
Combining (\ref{ev}) with (\ref{UL}) and taking $n$ to infinity, we obtain the results. This proof is complete.
\end{proof}

\begin{remark}\label{rmk4}
In Petersen and Salama's results \cite{petersen2020entropy}, they consider the golden-mean shift of finite type, that is, the adjacency matrix $A=\begin{bmatrix}
1&1\\
1&0
\end{bmatrix}
,$ on the $d$-tree $\mathcal{T}$ and show that the $n$-strip entropy $h^n(\mathcal{T},p)$ along $p=(\epsilon,f_1,f_1^2,...)$ converges to the topological entropy $h(\mathcal{T}).$ It is a special case of Theorem \ref{Thm: 1} in this paper. First, we claim that the adjacency matrix in Lemma \ref{lemma 1} is
\[
R=\begin{bmatrix}
    \beta_{n-1}^{d-1}&\beta_{n-1}^{d-1}\\
    \beta_{n-1}(0)^{d-1}&0
\end{bmatrix}.
\]
Indeed, from the definition of adjacency matrix $A,$ we obtain
\begin{align*}
    \alpha_{n,m+1}(0)&=(\alpha_{n,m}(0)+\alpha_{n,m}(1))\beta_{n-1}^{d-1};\\
    \alpha_{n,m+1}(1)&=\alpha_{n,m}(0)\beta_{n-1}(0)^{d-1}.
\end{align*}
This implies 
\[
\begin{bmatrix}
    \alpha_{n,m+1}(0)\\
    \alpha_{n,m+1}(1)
\end{bmatrix}=\left(R\equiv\begin{bmatrix}
    \beta_{n-1}^{d-1}&\beta_{n-1}^{d-1}\\
    \beta_{n-1}(0)^{d-1}&0
\end{bmatrix}\right)\begin{bmatrix}
    \alpha_{n,m}(0)\\
    \alpha_{n,m}(1)
\end{bmatrix}.
\]
From the matrix $R$, we obtain the following inequality
\begin{equation}\label{rmk1}
\beta_{n-1}^{d-1}\left(L\equiv\begin{bmatrix}
r&r\\
r&0
\end{bmatrix}\right)\leq R\leq \beta_{n-1}^{d-1} \left(U\equiv \begin{bmatrix}
2&2\\
2&0
\end{bmatrix}\right),
\end{equation}
where $0
<r<\min\left\{\frac{\beta_{n-1(0)}}{\beta_{n-1}},\frac{\beta_{n-1(1)}}{\beta_{n-1}}\right\}.$ Applying the inequality (\ref{rmk1}), we obtain 
\begin{equation}\label{rmk2}
\frac{\log \beta_{n-1}^{d-1}\rho_L}{(k-1)\lvert\Delta_{n-1}\rvert}\leq
\frac{\log\rho_R}{(d-1)\lvert\Delta_{n-1}\rvert}\leq 
\frac{\log \beta_{n-1}^{d-1}\rho_U}{(d-1)\lvert\Delta_{n-1}\rvert}, 
\end{equation}
where $\rho_L,\rho_R$ and $\rho_U$ are the maximal eigenvalues with respect to the matrices $L,R$ and $U.$ Finally, applying the inequality (\ref{rmk2}), we obtain 
\[
\lim\limits_{n\rightarrow \infty}h^n(\mathcal{T},p)=\lim\limits_{n\rightarrow \infty}\frac{\log \rho_R}{(d-1)\lvert\Delta_{n-1}\rvert}=\lim\limits_{n\rightarrow \infty}\frac{\log \beta_{n-1}}{\lvert\Delta_{n-1}\rvert}=h(\mathcal{T}).
\]
\end{remark}

\section{Proof of Theorem \ref{Thm: 1-2}} \label{sub 3}
In this section, we prove the limit in (\ref{1}) holds true for any eventually periodic path $p\in\partial\Sigma_M^*$, where $\Sigma_M^*$ is a complete recursive tree that $M$ has exactly one full row. The following Theorem in \cite{ban2022Commutativity} characterized the complete recursive trees.
\begin{theorem}[Theorem 3.6, \cite{ban2022Commutativity}]\label{Thm: complete}
 $\mathcal{T}$ is a complete recursive tree if and only if the adjacency matrix $M$ has the following properties: 1. there exists a nonempty subset $\mathfrak{J}$ of $\{1,...,d\}$ such that the $i$th row of $M$ is a full row for all $i\in \mathfrak{J}$; and 2. After reordering the symbols $\{1,...,d\},~M_{ij}=0$ for all $i\in\{1,...,d\}\setminus \mathfrak{J}$ and $1\leq j\leq i\leq d.$ 
\end{theorem}

Having fixed the path $p=(p_0,p_1,p_2,...)\in\partial\Sigma^*_M,$ we say $p$ is \emph{periodic} with period $\ell\in\mathbb{N}$ if $p_{q\ell+i}=\overbrace{p_\ell\cdots p_\ell}^{q} p_i$ for all $i=0,...,\ell-1$ and $q\geq 0.$ If there exists $c\in\mathbb{N}$ such that $p_{q\ell+i+c}=p_c\overbrace{p_\ell\cdots p_\ell}^{q}p_i$ for all $i=0,...,\ell-1$ and $q\geq 0,$ then we say $p$ is \emph{eventually periodic} with period $\ell\in\mathbb{N}.$ 

Before we prove Theorem \ref{Thm: 1-2}, we need the following lemma, which gives the eventual periodic path $p\in\partial\Sigma^*_M$ impact on the strip entropy. 

\begin{lemma}\label{lemma 2}
Let $p=\{p_0,p_1,p_2,...\}\in\partial\Sigma_M^*$ be eventually periodic with period $\ell\in\mathbb{N},$ that is, there exists $c\in\mathbb{N}$ such that $p_{q\ell+i+c}=p_c\overbrace{p_\ell\cdots p_\ell}^{q}p_i$ for all $i=0,...,\ell-1$ and $q\geq 0.$ Then 
    \[
    \Lambda^n(p_{q\ell+i+c+1})=\Lambda^n(p_{c+i+1})
    \]
    for all $i=0,...,\ell -1$ and $q\geq 0.$ 
\end{lemma}

\begin{proof}
    Let $p=\{p_0,p_1,p_2,...\}\in\partial\Sigma_M^*$ and there exists $c\in\mathbb{N}$ such that $p_{q\ell+i+c}=p_c\overbrace{p_\ell\cdots p_\ell}^{q}p_i$ for all $i=0,...,\ell-1$ and $q\geq 0.$ From the same observation in the golden-mean tree, the type of $\Lambda^n(p_j)$ is a function of $(f_s,f_{s'})$ that satisfies $(p_j,p_{j+1})=(p_{j-1}f_s,p_{j-1}f_sf_s')$ for all $f_sf_{s'}\in\Sigma_M^*$ and $j\geq 0.$ Since $p_{q\ell+i+c}=p_c\overbrace{p_\ell\cdots p_\ell}^{q}p_i$ for $q\geq 0,$ we obtain $p_{q\ell+i+c+1}=p_c\overbrace{p_\ell\cdots p_\ell}^{q}p_{i+1}$ and $p_{q\ell+i+c+2}=p_c\overbrace{p_\ell\cdots p_\ell}^{q}p_{i+2}.$ Then $(p_{q\ell+i+c+1},p_{q\ell+i+c+2})=(p_{q\ell+i+c}f_s,p_{q\ell+i+c}f_sf_{s'})$ if and only if $(p_{c+i+1},p_{c+i+2})=(p_{c+i}f_{s},p_{c+i}f_sf_{s'})$ for all $f_sf_{s'}\in\Sigma_M^*.$ This proof is complete. 
\end{proof}

\begin{proof}[Proof of Theorem \ref{Thm: 1-2}] Since $\Sigma_M^*$ is a complete recursive tree where $M$ has exactly one full row, by Theorem \ref{Thm: complete}, we may assume $M=(m_{st})_{1\leq s,t\leq d}\in \{0,1\}^{d\times d}$ with $m_{1t}=m_{n(n+1)}=m_{d1}=1$ for all $t=1,...,d$ and $n=2,...,d-1$; otherwise it is $0.$ Let $p=(p_0,p_1,p_2,...)\in\partial\Sigma_M^*$ be an eventually periodic path with period $\ell\in\mathbb{N}.$ Then there exists $c\in\mathbb{N}$ such that $p_{q\ell+i+c}=p_c\overbrace{p_\ell\cdots p_\ell}^{q}p_i$ for all $i=0,...,\ell-1$ and $q\geq 0.$ Applying Lemma \ref{lemma 2}, we obtain 
\begin{equation}\label{ep1}
\Lambda^n(p_{q\ell+i+c+1})=\Lambda^n(p_{c+i+1})
\end{equation}
for all $i=0,...,\ell-1$ and $q\geq 0.$ 

First, we claim that the type of $\Lambda^n(p_j)$ is finite, say $\{T_1^M,...,T_\xi^M\}.$ Indeed, since the type of $\Lambda^n(p_j)$ is a function of $(p_j,p_{j+1})$ and $d$ is finite, the cardinality of the set contains all pairs $(f_s,f_{s'})$, which satisfies that $(p_j,p_{j+1})=(p_{j-1}f_s,p_{j-1}f_sf_{s'})$ is finite. Thus, the type of $\Lambda^n(p_j)$ is finite. In fact, $\Lambda^n(p_j)$ is the union of sets in $\{\mathcal{B}_1^j,...,\mathcal{B}_d^j\},$ where \[\mathcal{B}_i^j=\{p_jf_iw\in\Sigma_M^*:\lvert p_jf_iw\rvert\leq j+n-1~\text{and}~p_jf_i\neq p_{j+1}\}.\] 

Second, we claim that there exists a $k\times k$ matrix $R$ such that 
\begin{equation}\label{e1}
(\alpha_{n,m+1}(1),...,\alpha_{n,m+1}(k))^T=R(\alpha_{n,m}(1),...,\alpha_{n,m}(k))^T.
\end{equation}
Indeed, without loss of generality, we consider the case 
\begin{equation}\label{111}
(p_m,p_{m+1},p_{m+2})=(p_{m-1}f_1,p_{m-1}f_1f_1,p_{m-1}f_1f_1f_1),
\end{equation}
otherwise, the other cases $(p_m,p_{m+1},p_{m+2})=(p_{m-1}f_s,p_{m-1}f_sf_{s'},p_{m-1}f_sf_{s'}f_{s''})$ consider the same process in the proof of case (\ref{111}) for all $f_sf_{s'}f_{s''}\in \Sigma_M^*.$
For $s\in \{1,...,k\},$ applying the same observation as in Lemma \ref{lemma 1}, we obtain
\begin{equation}\label{e2}
\alpha_{n,m+1}(s)=P_1P_2, 
\end{equation}
where $P_1=\lvert\mathcal{P}(\Lambda_m^n(p)\cup \{p_{m+1},\{t=(t_w)\in\mathcal{T}:t_{p_{m+1}}=s\}\})\rvert$ is the cardinality of the set that contains all labelings of $\mathcal{T}$ on $\Lambda_m^n (p)\cup \{p_{m+1}\}$ with $t_{p_{m+1}}=s,$ and $ P_2=\lvert\mathcal{P}(\Lambda^n(p_{m+1}),\{t=(t_w)\in \mathcal{T}:t_{p_{m+1}=s}\}) \rvert$ is the cardinality of the set that contains all labelings of $\mathcal{T}$ on $\Lambda^n(p_{m+1})$ with $t_{p_{m+1}}=s.$ From the choice of the path (\ref{111}), we obtain $\Lambda^n(p_{m+1})=\cup_{i=2}^{d}\mathcal{B}_i^{m+1}$ and each $\mathcal{B}_i^{m+1}$ is non-empty. Let $A^n=(a_{st}^{(n)})_{1\leq s,t\leq k},$ then we obtain the values of 
\begin{equation}\label{222}
P_1=a_{1s}\alpha_{n,m}(1)+a_{2s}\alpha_{n,m}(2)+\cdots +a_{ks}\alpha_{n,m}(k)
\end{equation}
and 
\begin{equation}\label{333}
P_2= \left(\sum_{\ell=1}^k a_{s\ell}^{(d)}\beta_{n-d}(\ell)\right)\left(\sum_{\ell=1}^k a_{s\ell}^{(d-1)}\beta_{n-(d-1)}(\ell)\right)\cdots \left(\sum_{\ell=1}^k a_{s\ell}^{(2)}\beta_{n-2}(\ell)\right).
\end{equation}
Combining the equations (\ref{e2}), (\ref{222}) and (\ref{333}), we obtain that the matrix $R$ satisfies the equality (\ref{e1}). In fact, $R=A^T\circ B,$ where $B=(b_{st})_{1\leq s,t\leq k}$ with $b_{st}=\prod_{j=2}^d\left(\sum_{\ell=1}^k a_{s\ell}^{(j)}\beta_{n-j}(\ell)\right).$ Moreover, the type of matrix $R$ is finite since the cardinality of $\{T_1^M,...,T_\xi^M\}$ is finite, say $R\in\{R_1,...,R_{\zeta}\}$ and each entry in $R_j$ is a product of the linear combination of $\beta_{n-i}(1),...,\beta_{n-i}(k)$ for all $i=1,...,d.$ 

Third, we claim that for each $R_j$ there exist $k\times k$ matrices $L_j$ and $U_j$ such that 
\[
\prod_{i=1}^d\beta_{n-i}^{t_{i,j}}L_j\leq R_j\leq \prod_{i=1}^d\beta_{n-i}^{t_{i,j}} U_j
\]
for some $1\leq t_{i,j}\leq d-1.$ Indeed, without loss of generality, we consider $R=A^T\circ B$ in the second claim. We choose 
$L=A^T\circ L'$ and $U=A^T\circ U',$ where $U'=(k^{d-1})_{1\leq s,t\leq k}$ and $L'=(r^{d-1})_{1\leq s,t\leq k}$ with $0<r<\min_{1\leq i\leq d}\left\{\frac{\beta_{n-i}(1)}{\beta_{n-i}},\frac{\beta_{n-i}(2)}{\beta_{n-i}},...,\frac{\beta_{n-i}(d)}{\beta_{n-i}}\right\}.$ Note that the constant $r$ exists since the adjacency matrix $A$ is primitive. Thus,
\[
\prod_{i=2}^d \beta_{n-i}L\leq R\leq \prod_{i=2}^d \beta_{n-i} U. 
\]

From (\ref{ep1}), there is a sequence of matrices $\{A_n\}_{n=1}^\infty$ with $A_n\in \{R_1,...,R_{\zeta}\}$ and $A_{q\ell+i+c+1}=A_{i+c+1}$ for all $i=0,...,\ell-1,~q\geq 0$ and $n\in\mathbb{N}.$ For convenience, we rewrite this sequence as $D_i=A_{q\ell+i+c}$ for all $i=1,...,\ell$ and $q\geq 0.$ 

Fourth, we claim that
\begin{equation}\label{c4}
h^n(\mathcal{T},p)=\frac{\log \rho_D}{\lambda^n(p_{c+1})+\cdots +\lambda^n(p_{c+\ell})},
\end{equation}
where $D=D_1\cdots D_\ell.$ Indeed, for fixed $n\in\mathbb{N}$ and $i\in\{1,...,\ell\},$
\begin{equation}\label{4}
\alpha_{n,q\ell+i+c}=\sum_{s,t=1}^k(A_1\cdots A_cD^qD_1\cdots D_i)_{st}.
\end{equation}
Recall that the adjacency matrix $A$ is primitive and each $R_j$ is of the form $A^T \circ B$ for some matrix $B$ whose entries in $B$ are a product of the linear combination of $\beta_{n-i}(1),...,\beta_{n-i}(k)$ for $i=1,...,d$ from the second claim. Therefore, $D$ is primitive. Then we obtain that there exist constants $c_0,d_0$ such that 
\begin{equation}\label{5}
c_0\rho_D^q\leq \sum_{s,t=1}^k(D^q)_{st}\leq d_0 \rho_D^q.
\end{equation}
From the second claim and the adjacency matrix $A$ is primitive, we obtain $A_1\cdots A_c$ and $D_1\cdots D_i$ being primitive. Then
\begin{equation}\label{6}
\sum_{s,t=1}^k(D^q)_{st}\leq\sum_{s,t=1}^k(A_1\cdots A_cD^qD_1\cdots D_i)_{st}\leq \sum_{s,t=1}^k(F^cD^qF^i)_{st},
\end{equation}
where each entry in $F$ is $\beta_{n-1}^{d-1}.$ This reduces the inequality (\ref{6}) into
\begin{equation}\label{7}
    \sum_{s,t=1}^k(D^q)_{st}\leq\sum_{s,t=1}^k(A_1\cdots A_cD^qD_1\cdots D_i)_{st}\leq \beta_{n-1}^{(d-1)(c+i)}k^{c+i}\sum_{s,t=1}^k(D^q)_{st}.
\end{equation}
Combining the inequality (\ref{5}) and (\ref{7}), we obtain that 
\begin{equation}\label{8}
c_0\rho_D^q\leq\sum_{s,t=1}^k(A_1\cdots A_cD^qD_1\cdots D_i)_{st}\leq \beta_{n-1}^{(d-1)(c+i)}k^{c+i}d_0\rho_D^q.
\end{equation}
Thus, combining (\ref{4}) and (\ref{8}), we obtain 
\begin{align*}
h^n(\mathcal{T},p)&=\lim_{q\rightarrow \infty}\frac{\log \alpha_{n,q\ell+i+c}}{\lambda^n(p_0)+\cdots +\lambda^n(p_{q\ell+i+c})}\\
&=\lim_{q\rightarrow \infty}\frac{\log \alpha_{n,q\ell+i+c}}{\lambda^n(p_0)+\cdots \lambda^n(p_c)+q(\lambda^n(p_{c+1})\cdots +\lambda^n(p_{c+\ell}))+\lambda^n(p_1)+\cdots +\lambda^n(p_i)} \\
&=\frac{\log \rho_D}{\lambda^n (p_{c+1})+\cdots +\lambda^n(p_{c+\ell})}.
\end{align*}

Finally, we prove that 
\[
\lim_{n\rightarrow \infty}\frac{\log \rho_D}{\lambda^n(p_{c+1})+\cdots +\lambda^n(p_{c+\ell})}=h(\mathcal{T},p).
\]
Indeed, without loss of generality, we consider the period $\ell=2$ and the path
\[
(p_c,p_{c+1},p_{c+2})=(p_c,p_cf_1,p_cf_1f_d).
\]
Applying the second claim, we obtain $D=D_1D_2$, $D_1=A^T\circ B$ and $D_2=A^T,$ where 
$B=(b_{st})_{1\leq s,t\leq k}$ with $b_{st}=\prod_{j=1}^{d-1}\left(\sum_{\ell=1}^k a_{s\ell}^{(j)}\beta_{n-j}(\ell)\right).$ Applying the third claim, we obtain that there exist $k\times k$ matrices $L=A^T\circ L'$ and $U=A^T\circ U',$ where $L'=(r^{d-1})_{1\leq s,t\leq k}$ with $0<r<\min_{1\leq i\leq d}\left\{\frac{\beta_{n-i}(1)}{\beta_{n-i}},\frac{\beta_{n-i}(2)}{\beta_{n-i}},...,\frac{\beta_{n-i}(d)}{\beta_{n-i}}\right\}$ and $U'=(k^{d-1})_{1\leq s,t\leq k}$ such that 
\begin{equation}\label{9}
\prod_{j=1}^{d-1} \beta_{n-j}LA^T\leq D=D_1D_2\leq \prod_{j=1}^{d-1} \beta_{n-j} UA^T. 
\end{equation}
Then, since the adjacency matrix $A$ is primitive, the inequality (\ref{9}) leads to
\begin{equation}\label{10}
    \prod_{j=1}^{d-1}\beta_{n-j} \rho_{LA^T}\leq \rho_D\leq \prod_{j=1}^{d-1}\beta_{n-j}\rho_{UA^T}.
\end{equation}
Combining (\ref{c4}) and the inequality (\ref{10}), we conclude 
\begin{align*}
\lim_{n\rightarrow \infty}h^n(\mathcal{T},p)&=\lim_{n\rightarrow \infty}\frac{\log \rho_D}{\lambda^n(p_{c+1})+\cdots +\lambda^n(p_{c+\ell})}\\
&=\lim_{n\rightarrow \infty}\frac{\sum_{j=1}^{d-1}\log \beta_{n-j}}{\lambda^n(p_{c+1})+\lambda^n(p_{c+2})}\\
&=\lim_{n\rightarrow \infty}\frac{\sum_{j=1}^{d-1}\log \beta_{n-j}}{\sum_{j=1}^{d-1}\lvert\Delta_{n-j}\rvert}=h(\mathcal{T},p).
\end{align*}
The proof is complete.
\end{proof}

\bibliographystyle{amsplain}
\bibliography{ban}

\end{document}